\documentclass[preprint,12pt,1p]{elsarticle}
\makeatother
\usepackage{graphics}
\usepackage{epsfig}
\usepackage{mathptmx}
\usepackage{amsmath}
\usepackage{amssymb}
\usepackage{hyperref}
\usepackage{fullpage}
\usepackage{amsthm}
\usepackage{lineno}

\theoremstyle{plain}
\newtheorem{theorem}{Theorem}[section]
\newtheorem{corollary}[theorem]{Corollary}
\newtheorem{lemma}[theorem]{Lemma}

\theoremstyle{definition}
\newtheorem{definition}{Definition}[section]
\theoremstyle{example}

\theoremstyle{remark}

\newlength{\defbaselineskip}
\makeatletter
\def\ps@pprintTitle{
  \let\@oddhead\@empty
  \let\@evenhead\@empty
  \let\@oddfoot\@empty
  \let\@evenfoot\@oddfoot
}
\journal{Computer aided geometric design}
\begin{document}
\begin{frontmatter}
\title{On some  inequalities for submanifolds of Bochner Kaehler manifolds}
\author[n1]{Mehraj Ahmad Lone\corref{cor1}}
\ead{mehraj.jmi@gmail.com}
\author[n2]{Mohammed Jamali}
\author[n3]{Mohammad Hasan Shahid}
\address[n1]{Department of Mathematics, Central University of Jammu, Jammu-180011, India.}
\address[n2]{Department of Mathematics, Al-Falah University, Haryana-121004, India.}
\address[n3]{Department of Mathematics, Jamia Millia Islamia, New Delhi-110 025, India.}
\cortext[cor1]{Corresponding author}
\begin{abstract}
B. Y. Chen established sharp inequalities between certain Riemannian invariants and the squared mean curvature for submanifolds in real space form as well as in complex space form. In this paper we generalize Chen inequalities for submanifolds of Bochner Kaehler manifolds. Moreover, we consider CR-warped product submanifolds  of Bochner Kaehler manifold and establish an inequality for scalar curvature.
\end{abstract}
\begin{keyword}
\texttt Bochner Kaehler manifold, CR-warped product, slant submanifolds, Einstein manifold, Chen inequality.

\end{keyword}
\end{frontmatter}

\section{Introduction}
In \cite{biha4}, B. Y. Chen established   sharp inequality for a submanifold in a real space form involving intrinsic invariants of the submanifolds and squared mean curvature, the main extrinsic invariant and in \cite{biha2}, B. Y. Chen obtained the same inequality  for complex space form.  After that many research articles \cite{biha5, biha6, biha1} have been published by different authors for different submanifolds and ambient spaces in complex as well as in contact version. In this article we obtain these inequalities for submanifolds in Bochner Kaehler manifold.

 In \cite{bishopwrap} Bishop and O'Neil initiated the thoery of warped product submanifold as a generalization of pseudo-Riemannian product manifold. In \cite{chen warp} Chen introduced the notion of CR-warped products. In This paper we study the CR-warped product submanifolds of Bochner Kaehler manifolds.
\newline
\section{Preliminaries}
Let $\mathcal{W}$ be a $n$-dimensional submanifold of a Bochner Kaehler manifold $\overline{\mathcal{W}}$ of dimension $2m$. Let $\nabla$ and $\overline{\nabla}$ be the  Levi-Civita connection on $\mathcal{W}$ and $\overline{\mathcal{W}}$ respectively. Let $J$ be the complex structure on $\overline{\mathcal{W}}$. Then the Gauss and Weingarten formulas are given respectively by
\begin{eqnarray}\label{a1}
\overline{\nabla}_{X}Y = \nabla_{X}Y + \omega(X,Y),
\end{eqnarray}
\begin{eqnarray}\label{a2}
\overline{\nabla}_{X}V =  - B_{V}X + \nabla_{X}^{\perp}Y ,
\end{eqnarray}
for all $X, Y$ tangent to  $\mathcal{W}$ and vector field $V$ normal to $\mathcal{W}$. Where $\omega$, $\nabla_{X}^{\perp}$, $B_{V}$ denotes the second fundamental form, normal connection and the shape operator respectively. The second fundamental form and the shape operator are related by
\begin{eqnarray}\label{a3}
g(\omega(X,Y), V) = g(B_{V}X, Y).
\end{eqnarray}
Let $R$ be the curvature tensor of $\mathcal{W}$, Then the Gauss equation is given by \cite{biha4}
\begin{eqnarray*}\label{a4}
\overline{R}(X,Y,Z,W) = R(X,Y,Z,W) + g(\omega(X,W),\omega(Y,Z)) - g(\omega(X,Z),\omega(Y,W))
\end{eqnarray*}
for any vector fields $X$, $Y$, $Z$, $W$ tangent to $\mathcal{W}$.

The curvature tensor of a Bochner Kaehler manifold $\overline{\mathcal{W}}$ is given by \cite{biha7}
\begin{eqnarray}\label{a5}
\overline{R}(X,Y,Z,W)\nonumber &=& L(Y,Z)g(X,W) - L(X,Z)g(Y,W) + L(X,W)g(Y,Z) \\ \nonumber && - L(Y,W)g(X,Z) + M(X,W)g(JX,W) - M(X,Z)g(JY,W)\\ \nonumber && + M(X,W)g(JY,Z) - M(Y,W)g(JX,Z) \\  && - 2M(X,Y)g(JZ,W) - 2M(Z,W)g(JX,Y)
\end{eqnarray}

where
\begin{eqnarray}\label{a6}
L(Y,Z) = \frac{1}{2n+4}Ric(Y,Z) - \frac{\rho}{2(2n+2)(2n+4)}g(Y,Z),
\end{eqnarray}
\begin{eqnarray}\label{a7}
M(Y,Z)  = -L(Y,JZ),
\end{eqnarray}
\begin{eqnarray}\label{a8}
L(Y,Z)  = L(Z,Y),\hspace{1cm} L(Y,Z) = L(JY, JZ), \hspace{1cm} L(Y, JZ) = -L(JY,Z),
\end{eqnarray}
$Ric$ and $\rho$ are the Ricci tensor and scalar curvature of  $\mathcal{W}$.

Let $x\in \mathcal{W}$ and $\{e_{1}, ... , e_{n}\}$ be an orthonormal basis of the tangent space $T_{x}\mathcal{W}$ and  $\{e_{n+1}, ... , e_{2m}\}$ be the orthonormal basis  of $T^{\perp}\mathcal{W}$. We denote by $\mathcal{H}$, the mean curvature vector at $x$, that is
\begin{eqnarray}\label{a9}
\mathcal{H}(x) = \frac{1}{n}\sum_{i=1}^{n}\omega(e_{i},e_{i}),
\end{eqnarray}
Also, we set
\begin{eqnarray*}\label{a10}
\omega_{ij}^{r} = g(\omega(e_{i},e_{j}),e_{r}), \hspace{1cm} i,j \in \{ 1, ... , n\},\hspace{.3cm} r \in \{n+1, ... ,2m\}
\end{eqnarray*}
and
\begin{eqnarray}\label{a11}
\|\omega\|^{2} = \sum_{i,j=1}^{n}(\omega(e_{i},e_{j}), \omega(e_{i},e_{j})).
\end{eqnarray}
For any $x \in \mathcal{W}$ and $X \in T_{x}\mathcal{W}$, we put $JX = TX + FX$, where $TX$ and $FX$ are the tangential and normal components of $JX$, respectively.

 We denote by
\begin{eqnarray*}\label{a12}
\|T\|^{2} = \sum_{i,j=1}^{n}g^{2}(Te_{i}, e_{j}).
\end{eqnarray*}
Let $\mathcal{W}$ be a Riemannian manifold. Denote by $\mathcal{K}(\pi)$ the sectional curvature of $\mathcal{W}$ of the plane section $\pi \subset T_{x}\mathcal{W}, x\in \mathcal{W}$.  The scalar curvature $\rho$ for an orthonormal basis$ \{e_{1}, e_{2}, ..., e_{n}\}$ of the tangent space $T_{x}\mathcal{W}$  at $x$ is defined by
\begin{eqnarray*}\label{a121}
\rho(x) = \sum_{i<j} K(e_{i}\wedge e_{j}).
\end{eqnarray*}

\begin{lemma}\cite{biha4}
Let $n \geq 2$ and $x_{1}, x_{2}, ... ,x_{n}$, b be real numbers such that
\begin{eqnarray*}\label{l3}
( \sum_{i=1}^{n} x_{i} )^{2} = (n-1)( \sum_{i=1}^{n} x_{i}^{2} + b)
\end{eqnarray*}
then $2x_{1}x_{2} \geq b$, with equality holds if and only if
$$ x_{1} + x_{2} = x_{3} = ... = x_{n}.$$
\end{lemma}
In \cite{biha11} A. Bejancu introduced the notion of CR-submanifolds, which is the generalization of invariant and anti-invariant submanifolds. In  \cite{biha3} B. Y. Chen introduced  the  notion of slant submanifolds as a generalization of CR-submanifolds.
\begin{definition}
A submanifold $\mathcal{W}$ of a Bochner Kaehler manifold $\overline{\mathcal{W}}$ is said to be a slant submanifold if for any $x \in \mathcal{W}$ and $X \in T_{x}\mathcal{W}$, the angle between $JX$ and $T_{x}\mathcal{W}$ is constant, i.e., the angle does not depend on the choice of $x \in \mathcal{W}$ and $X \in T_{x}\mathcal{W}$. The angle $\theta \in [0, \frac{\pi}{2}]$ is called the slant angle of $\mathcal{W}$ in $\overline{\mathcal{W}}$.

Invariant and anti-invariant submanifolds are the slant submanifolds with slant angle $\theta = 0$ and $\theta =\frac{\pi}{2}$ respectively and when $0< \theta < \frac{\pi}{2}$, then slant submanifold is called proper slant submanifold.
\end{definition}
\begin{definition}
Let $(N_{1}, g_{1})$ and $(N_{2}, g_{2})$ be two Riemannian manifolds and f, a positive
differentiable function on $N_{1}$. The warped product of $N_{1}$ and $N_{2}$ is the Riemannian
manifold $M = N_{1} \times N_{2} = (N_{1} \times N_{2}, g)$, where $g = g_{1} + f^{2}g_{2}$
\end{definition}
\begin{definition}
A Riemannian manifold $\mathcal{W}$ is said to be Einstein manifold if the Ricci tensor is proportional to the metric tensor, that is, $Ric(X,Y) = \lambda g(X,Y)$ for some  constant $\lambda$.
\end{definition}

\section{B. Y. Chen inequalities}
In this section, we obtain B. Y. Chen inequalities for submanifolds of a Bochner Kaehler manifolds.

First we have,
\begin{theorem}
Let $\mathcal{W}$ be a submanifold of  a Bochner Kaehler manifold $\overline{\mathcal{W}}$. Then, for each point $x \in \mathcal{W}$ and each plane section $\pi \subset T_{x}\mathcal{W}$, we have
\begin{eqnarray}\label{t1}
\nonumber \mathcal{K}(\pi) \geq \bigg(\frac{5n^{2}+31n+26+3\|T\|^{2}}{2(2n+2)(2n+4)}\bigg)\rho - \frac{n^{2}(n-2)}{2(n-1)}\|\mathcal{H}\|^{2} - \frac{6}{2(2n+4)}Ric(e_{i},Je_{j})g(e_{i},Je_{j}).\\
\end{eqnarray}
Equality  holds if and only if there exists an orthonormal basis $\{e_{1} ,e_{2}, ... , e_{n}\}$ of $T_{x}\mathcal{W}$ and orthonormal basis $\{e_{n+1}, e_{n+2} , ... , e_{2m}\}$ of $T^{\perp}\mathcal{W}$ such that the shape operators takes the following forms

\begin{eqnarray}\label{t2}
B_{n+1} =
 \begin{pmatrix}
  \alpha & 0 & 0 & \cdots & 0 \\
  0 & \beta & 0 & \cdots & 0 \\
  0 & 0 & \xi &\cdots & 0 \\
  \vdots & \vdots & \vdots  & \ddots & \vdots  \\
  0 & 0 & 0 & \cdots & \xi
 \end{pmatrix}  , \alpha+\beta = \xi
\end{eqnarray}
and
\begin{eqnarray}\label{t3}
B_{r} =
 \begin{pmatrix}
  \omega_{11}^{r} &  \omega_{12}^{r} & 0 & \cdots & 0 \\
   \omega_{12}^{r} & - \omega_{11}^{r} & 0 & \cdots & 0 \\
  0 & 0 & 0 &\cdots & 0 \\
  \vdots & \vdots & \vdots  & \ddots & \vdots  \\
  0 & 0 & 0 & \cdots & 0
 \end{pmatrix}  , r = n+2, ..., 2m.
\end{eqnarray}
\end{theorem}
\begin{proof}
Using Gauss equation, the Riemannian curvature tensor of $\mathcal{W}$ is given  by
\begin{eqnarray*}\label{p1}
\nonumber R(X,Y,Z,W) &=& L(Y,Z)g(X,W) - L(X,Z)g(Y,W) + L(X,W)g(Y,Z)
                          \nonumber \\ && - L(Y,W)g(X,Z)  + M(Y,Z)g(JX,W) - M(X,Z)g(JY,W)
                          \nonumber \\ && - M(X,W)g(JY,Z) - M(Y,W)g(JX,Z)  - 2M(X,Y)(JZ,W)
                           \nonumber \\ && - 2M(Z,W)g(JX,Y) + g(\omega(X,W), \omega(Y,Z)) - g(\omega(X,Z), \omega(Y,W))
\end{eqnarray*}
for any  X, Y, Z, W $\in$ T$\mathcal{W}$.
\begin{eqnarray*}\label{p2}
\sum_{i,j} R(e_{i}, e_{j}, e_{j}, e_{i}) &=& L(e_{j},e_{j})g(e_{i},e_{i}) - L(e_{i},e_{j})g(e_{j},e_{i}) + L(e_{i},e_{i})g(e_{j},e_{j})
                          \nonumber \\ && - L(e_{j},e_{i})g(e_{i},e_{j})  + M(e_{j},e_{j})g(Je_{i},e_{i}) - M(e_{i},e_{j})g(Je_{j},e_{i})
                          \nonumber \\ && - M(e_{i},e_{i})g(Je_{j},e_{j}) - M(e_{j},e_{i})g(Je_{i},e_{j})  - 2M(e_{i},e_{j})(Je_{j},e_{i})
                           \nonumber \\ && - 2M(e_{j},e_{i})g(Je_{i},e_{j}) + g(\omega(e_{i},e_{i}), \omega(e_{j},e_{j})) - g(\omega(e_{i},e_{j}), \omega(e_{j},e_{i}))
\end{eqnarray*}
\begin{eqnarray}\label{p3}
 \hspace{3cm}                       &=& L(e_{j},e_{j})g(e_{i},e_{i}) - L(e_{i},e_{j})g(e_{j},e_{i}) + L(e_{i},e_{i})g(e_{j},e_{j})
                         \nonumber \\ && - L(e_{j},e_{i})g(e_{i},e_{j})  - L(e_{j},Je_{j})g(Je_{i},e_{i}) + L(e_{i},Je_{j})g(Je_{j},e_{i})
                          \nonumber \\ && + L(e_{i},Je_{i})g(Je_{j},e_{j}) + L(e_{j},Je_{i})g(Je_{i},e_{j}) + 2L(e_{i},Je_{j})(Je_{j},e_{i})
                           \nonumber \\ && + 2L(e_{j},Je_{i})g(Je_{i},e_{j}) + g(\omega(e_{i},e_{i}), \omega(e_{j},e_{j})) - g(\omega(e_{i},e_{j}), \omega(e_{j},e_{i})).\nonumber \\
\end{eqnarray}
Using (\ref{a8}), (\ref{a9}) and (\ref{a11}) in (\ref{p3}), we have
\begin{eqnarray*}\label{p4}
\sum_{i,j} R(e_{i}, e_{j}, e_{j}, e_{i})&=&  2nL(e_{i}, e_{i}) - 2L(e_{i}, e_{j})g(e_{i}, e_{j}) + 6L(e_{i}, Je_{j})g(e_{i}, Je_{j})
                                    \nonumber \\ &&  + n^{2}\|\mathcal{H}\|^{2} - \|\omega\|^{2}.
\end{eqnarray*}
 Which  simplifies to,
\begin{eqnarray}\label{p5}
2\rho  &=&  2(n-1)L(e_{i}, e_{i}) + 6L(e_{i}, Je_{j})g(e_{i}, Je_{j}) + n^{2}\|\mathcal{H}\|^{2} - \|\omega\|^{2}.
\end{eqnarray}
 Combining (\ref{a6}) and (\ref{p5}), we have
\begin{eqnarray*}\label{p6}
\hspace{2cm}2\rho  &=& \frac{2(n-1)}{2n+4}Ric(e_{i}, e_{i}) - \frac{2(n-1)\rho}{2(2n+2)(2n+4)}g(e_{i}, e_{i})
        \nonumber \\ &&   + \frac{6}{2n+4} Ric(e_{i}, Je_{j})g(e_{i}, Je_{j}) - \frac{6 \rho}{2(2n+2)(2n+4)} g(e_{i}, Je_{j})g(e_{i}, Je_{j})
        \nonumber \\ &&   + n^{2}\|\mathcal{H}\|^{2} - \|\omega\|^{2}.
\end{eqnarray*}
or
\begin{eqnarray*}\label{p9}
\hspace{2cm}2\rho  &=& \frac{6n^{2} + 2n - 8 - 6\|T\|^{2}}{2(2n + 2)(2n + 4)}\rho + \frac{6}{2n+4} Ric(e_{i}, Je_{j})g(e_{i}, Je_{j})
       \nonumber \\ && + n^{2}\|\mathcal{H}\|^{2} - \|\omega\|^{2}.
\end{eqnarray*}
or
\begin{eqnarray*}\label{p10}
(2 - \frac{6n^{2} + 2n - 8 - 6\|T\|^{2}}{2(2n + 2)(2n + 4)})\rho = \frac{6}{2n + 4}Ric(e_{i}, Je_{j})g(e_{i}, Je_{j}) + n^{2}\|\mathcal{H}\|^{2} - \|\omega\|^{2}.
\end{eqnarray*}
Denoting by
\begin{eqnarray*}\label{p11}
\epsilon = (2 - \frac{6n^{2} + 2n - 8 - 6\|T\|^{2}}{2(2n + 2)(2n + 4)})\rho - \frac{n^{2}(n - 2)}{n - 1}\|\mathcal{H}\|^{2} - \frac{6}{2n + 4}Ric(e_{i}, Je_{j})g(e_{i}, Je_{j}),
\end{eqnarray*}
we obtain
\begin{eqnarray*}\label{p12}
\epsilon =  n^{2}\|\mathcal{H}\|^{2} - \|\omega\|^{2} - \frac{n^{2}(n - 2)}{n - 1}\|\mathcal{H}\|^{2}.
\end{eqnarray*}
or
\begin{eqnarray}\label{p13}
n^{2}\|\mathcal{H}\|^{2} = (n - 1)(\epsilon + \|\omega\|^{2}).
\end{eqnarray}

For chosen orthonormal basis, the above equation takes the form

\begin{eqnarray}\label{p14}
(\sum_{i=1}^{n}\omega_{ii}^{n+1})^{2} = (n-1)\left[\sum_{i=1}^{n}(\omega_{ii}^{n+1})^{2} + \sum_{i\ne j}(\omega_{ij}^{n+1})^{2} + \sum_{r=n+1}^{2m}\sum_{i,j=1}^{n}(\omega_{ij})^{2} + \epsilon \right].
\end{eqnarray}

Using lemma 1 in (\ref{p14}), we have

\begin{eqnarray}\label{p15}
2\omega_{11}^{n+1}\omega_{22}^{n+1} \geq \sum_{i\ne j}(\omega_{ij}^{n+1})^{2} + \sum_{r=n+1}^{2m}\sum_{i,j=1}^{n}(\omega_{ij}^{r})^{2} + \epsilon.
\end{eqnarray}

On the other hand, from Gauss equation we obtain

\begin{eqnarray}\label{p16}
\mathcal{K}(\pi) = L(e_{2}, e_{2}) + L(e_{1},e_{1}) + g(\omega(e_{1}, e_{1}), \omega(e_{2},e_{2}) - g(\omega(e_{1},e_{2}), \omega(e_{2},e_{1})).
\end{eqnarray}
Combing (\ref{a6}) and (\ref{p16}), we derive
\begin{eqnarray}\label{p17}
\mathcal{K}(\pi) = \frac{4n+3}{(2n+2)(2n+4)}\rho + \omega_{11}^{n+1}\omega_{22}^{n+1} + \sum_{r=n+2}^{2m}\omega_{11}^{r}\omega_{22}^{r} - \sum_{r=n+1}^{2m}(\omega_{12}^{r})^{2}.
\end{eqnarray}

Incorporating (\ref{p15}) in (\ref{p17}), we arrive at the inequality
\begin{eqnarray*}\label{p18}
\mathcal{K}(\pi) &\geq& \frac{1}{2}\sum_{i\ne j}(\omega_{ij}^{n+1})^{2} + \frac{1}{2}\sum_{r=n+1}^{2m}\sum_{i,j=1}^{n}(\omega_{ij}^{r})^{2} + \frac{1}{2}\epsilon  \nonumber \\ && + \frac{4n+3}{(2n+2)(2n+4)}\rho + \sum_{r=n+2}^{2m}\omega_{11}^{r}\omega_{22}^{r} - \sum_{r=n+1}^{2m}(\omega_{12}^{r})^{2}.
\end{eqnarray*}
Which implies that
\begin{eqnarray*}\label{p19}
\mathcal{K}(\pi) \geq \frac{4n+3}{(2n+2)(2n+4)}\rho + \frac{1}{2}\epsilon.
\end{eqnarray*}
or

\begin{eqnarray}\label{p20}
\nonumber \mathcal{K}(\pi) \geq (\frac{5n^{2}+31n+26+3\|T\|^{2}}{2(2n+2)(2n+4)})\rho - \frac{n^{2}(n-2)}{2(n-1)}\|\mathcal{H}\|^{2} - \frac{6}{2(2n+4)}Ric(e_{i},Je_{j})g(e_{i},Je_{j}). \\
\end{eqnarray}
If the equality in (\ref{t1}) at a point $p$ holds, then the inequality (\ref{p20}) become equality. In this case, we have

$\begin{cases}
\omega_{1j}^{n+1} = \omega_{2j}^{n+1} = \omega_{ij}^{n+1} = 0, \hspace{1cm} i \neq j >2,\\
\omega_{ij}^{r} = 0, \forall i \neq j,\hspace{.5cm}i,j = 3, ..., 2m, \hspace{.5cm}r= n+1, ..., 2m,\\
\omega_{11}^{r}+\omega_{22}^{r} = 0, \forall r=n+2, ..., 2m,\hspace{2cm}\\
\omega_{11}^{n+2}+\omega_{22}^{n+1} = ... = \omega_{11}^{m}+\omega_{22}^{m}=0.\\
\end{cases}$

Now, if we choose $e_{1},e_{2}$ such that $\omega_{12}^{n+1}$= 0 and we denote by $\alpha = \omega_{11}^{r}, \beta = \omega_{22}^{r}$, $\xi = \omega_{33}^{n+1} = ... = \omega_{33}^{r}$. Therefore by choosing the suitable orthonormal basis the shape operators take the desired forms.
\end{proof}
We conclude the following corollary from this theorem.
\begin{corollary}
Let $\mathcal{W}$ be a submanifold of  a Bochner Kaehler manifold $\overline{\mathcal{W}}$ which is Einstein. Then, for each point $x \in \mathcal{W}$ and each plane section $\pi \subset T_{x}\mathcal{W}$, we have
\begin{eqnarray*}\label{p21}
\mathcal{K}(\pi) \geq (\frac{5n^{2}+31n+26+3\|T\|^{2}}{2(2n+2)(2n+4)})\rho - \frac{n^{2}(n-2)}{2(n-1)}\|\mathcal{H}\|^{2} - \frac{6\lambda}{2(2n+4)} \|T\|^{2}.
\end{eqnarray*}
The equality at a point $x \in \mathcal{W}$ holds iff there exists an orthonormal basis $\{e_{1},e_{2}, ..., e_{n}\}$ of $ T_{x}\mathcal{W}$ and orthonormal basis  $\{e_{n+1},e_{n+2}, ..., e_{2m}\}$ of $T^{\perp}\mathcal{W}$ such that shape operators of $\mathcal{W}$ in $\overline{\mathcal{W}}$ at $x$ have the forms (\ref{t2}) and (\ref{t3}).
\end{corollary}
Similarly, in case if $\mathcal{W}$ is a slant submanifold of  a Bochner Kaehler manifold $\overline{\mathcal{W}}$. We have the following theorem
\begin{theorem}
Let $\mathcal{W}$ be a slant submanifold of  a Bochner Kaehler manifold $\overline{\mathcal{W}}$. Then, for each point $x \in \mathcal{W}$ and each plane section $\pi \subset T_{x}\mathcal{W}$, we have
\begin{eqnarray*}\label{p22}
\nonumber \mathcal{K}(\pi) \geq (\frac{5n^{2}+31n+26+3cos^{2}\theta}{2(2n+2)(2n+4)})\rho - \frac{n^{2}(n-2)}{2(n-1)}\|\mathcal{H}\|^{2} - \frac{6}{2(2n+4)}Ric(e_{i},Je_{j})cos\theta.\\
\end{eqnarray*}
Equality  holds if and only if there exists an orthonormal basis $\{e_{1} ,e_{2}, ... , e_{n}\}$ of $T_{x}\mathcal{W}$ and orthonormal basis $\{e_{n+1}, e_{n+2} , ... , e_{2m}\}$ of $T^{\perp}\mathcal{W}$ such that the shape operator takes the following forms

\begin{eqnarray}\label{p23}
B_{n+1} =
 \begin{pmatrix}
  \alpha & 0 & 0 & \cdots & 0 \\
  0 & \beta & 0 & \cdots & 0 \\
  0 & 0 & \xi &\cdots & 0 \\
  \vdots & \vdots & \vdots  & \ddots & \vdots  \\
  0 & 0 & 0 & \cdots & \xi
 \end{pmatrix}  , \alpha+\beta = \xi
\end{eqnarray}
and
\begin{eqnarray}\label{p24}
B_{r} =
 \begin{pmatrix}
  \omega_{11}^{r} &  \omega_{12}^{r} & 0 & \cdots & 0 \\
   \omega_{12}^{r} & - \omega_{11}^{r} & 0 & \cdots & 0 \\
  0 & 0 & 0 &\cdots & 0 \\
  \vdots & \vdots & \vdots  & \ddots & \vdots  \\
  0 & 0 & 0 & \cdots & 0
 \end{pmatrix}  , r = n+2, ..., 2m.
\end{eqnarray}
\end{theorem}
From this theorem, following corollaries can be easily deduced.
\begin{corollary}
Let $\mathcal{W}$ be a slant submanifold of  a Bochner Kaehler manifold $\overline{\mathcal{W}}$, which is Einstein . Then, for each point $x \in \mathcal{W}$ and each plane section $\pi \subset T_{x}\mathcal{W}$, we have
\begin{eqnarray*}\label{p25}
\nonumber \mathcal{K}(\pi) \geq (\frac{5n^{2}+31n+26+3cos^{2}\theta}{2(2n+2)(2n+4)})\rho - \frac{n^{2}(n-2)}{2(n-1)}\|\mathcal{H}\|^{2} - \frac{6\lambda}{2(2n+4)}cos^{2}\theta.\\
\end{eqnarray*}
The equality  holds at a point $x \in \mathcal{W}$ if and only if there exists an orthonormal basis $\{e_{1},e_{2}, ..., e_{n}\}$ of $ T_{x}\mathcal{W}$  and orthonormal basis  $\{e_{n+1},e_{n+2}, ..., e_{2m}\}$ of $T^{\perp}\mathcal{W}$ such that shape operators of $\mathcal{W}$ in $\overline{\mathcal{W}}$ at $x$ have the forms (\ref{p23}) and (\ref{p24}).
\end{corollary}
\begin{corollary}
Let $\mathcal{W}$ be a invariant submanifold of  a Bochner Kaehler manifold $\overline{\mathcal{W}}$ . Then, for each point $x\in \mathcal{W}$ and each plane section $\pi \subset T_{x}\mathcal{W}$, we have
\begin{eqnarray*}\label{p26}
 \mathcal{K}(\pi) \geq (\frac{5n^{2}+31n+26+3}{2(2n+2)(2n+4)})\rho - \frac{n^{2}(n-2)}{2(n-1)}\|\mathcal{H}\|^{2} - \frac{6}{2(2n+4)}Ric(e_{i}, Je_{j}).
\end{eqnarray*}
The equality  at a point $x\in \mathcal{W}$ holds iff there exists an orthonormal basis $\{e_{1},e_{2}, ..., e_{n}\}$ of $ T_{x}\mathcal{W}$ and orthonormal basis  $\{e_{n+1},e_{n+2}, ..., e_{2m}\}$ of $T^{\perp}\mathcal{W}$ such that shape operators of $\mathcal{W}$ in $\overline{\mathcal{W}}$ at $x$ have the forms (\ref{p23}) and (\ref{p24}).
\end{corollary}
\begin{corollary}
Let $\mathcal{W}$ be a anti-invariant submanifold of  a Bochner Kaehler manifold $\overline{\mathcal{W}}$ . Then, for each point $x \in \mathcal{W}$ and each plane section $\pi \subset T_{x}\mathcal{W}$, we have
\begin{eqnarray*}\label{p26}
 \mathcal{K}(\pi) \geq (\frac{5n^{2}+31n+26}{2(2n+2)(2n+4)})\rho - \frac{n^{2}(n-2)}{2(n-1)}\|\mathcal{H}\|^{2}.
\end{eqnarray*}
The equality  at a point $x \in \mathcal{W}$ holds iff there exists an orthonormal basis $\{e_{1},e_{2}, ..., e_{n}\}$ of $ T_{x}\mathcal{W}$ and orthonormal basis  $\{e_{n+1},e_{n+2}, ..., e_{2m}\}$ of $T^{\perp}\mathcal{W}$ such that shape operators of $\mathcal{W}$ in $\overline{\mathcal{W}}$ at $x$ have the forms  (\ref{p23}) and (\ref{p24}).
\end{corollary}
\section{Warped product of CR-submanifolds of Bochner Kaehler manifolds}
Let $\mathcal{W} = \mathcal{W}_{T}\times_{f}\mathcal{W}_{\perp}$ be the warped product CR-submanifolds of Bochner Kaehler manifold $\overline{\mathcal{W}}$ such that the invariant distribution is $D = T\mathcal{W}_{T}$ and anti-invariant distribution is $D^{\perp} = T\mathcal{W}_{\perp}$, where $f: \mathcal{W}_{T} \longrightarrow \mathbb{R}$. Then the metric $g$ on $\mathcal{W}$ is given by \cite{chen warp}
\begin{equation*}\label{w1}
g(X,Y) =  \langle \pi_{*}X, \pi_{*}Y\rangle + (f\circ\pi)^{2}\langle \sigma_{*}X, \sigma_{*}Y\rangle
\end{equation*}
where $\pi$ and $\sigma$ are the projection maps from $\mathcal{W}$ onto $\mathcal{W}_{T}$ and $\mathcal{W}_{\perp}$ respectively.

%

It is easy to see that
\begin{eqnarray}\label{w6}
T\mathcal{W} = D \oplus D^{\perp}\hspace{.3cm} \text{and} \hspace{.3cm} T^{\perp}\mathcal{W} = JD^{\perp}\oplus \nu,
\end{eqnarray}
where $\nu$ is the orthogonal distribution to $JD^{\perp}$ in the normal bundle $T^{\perp}\mathcal{W}.$

From (\ref{w6}), we can write
\begin{eqnarray*}\label{w7}
\omega(X,Y) = \omega_{JD^{\perp}}(X,Y) + \omega_{\nu}(X,Y)
\end{eqnarray*}
Also for warped product submanifold $\mathcal{W}$ of $\overline{\mathcal{W}}$, we have \cite{chen warp}
\begin{eqnarray}\label{w8}
\nabla_{X}Z = X(logf)Z = \frac{X(f)}{f}Z
\end{eqnarray}
for any vector fields $X \in D$ and $Z \in D^{\perp}$.

Further, we can decompose  $(\overline{\nabla}_{X}J)Y$ into the tangential and normal components  as under
\begin{eqnarray}\label{w9}
(\overline{\nabla}_{X}J)Y = \mathcal{P}_{X}Y + \mathcal{Q}_{X}Y
\end{eqnarray}
where $\mathcal{P}_{X}Y$ and  $\mathcal{Q}_{X}Y$ denotes the tangential and normal components of $(\overline{\nabla}_{X}J)Y$

First we prove the following lemma
\begin{lemma}
Let $\mathcal{W} = \mathcal{W}_{T}\times_{f}\mathcal{W}_{\perp}$ be a CR-warped product submanifold of a Bochner Kaehler manifold $\overline{\mathcal{W}}$. Then we have
\begin{equation*}
\omega_{JD^{\perp}}(JX,Z) = J\mathcal{P}_{Z}JX + X(log f)JZ
\end{equation*}
\begin{equation*}
g(\mathcal{P}_{Z}JX, W) = g(\mathcal{Q}_{Z}X, JW)
\end{equation*}
 and
\begin{equation*}
g(\omega(JX,Z), J\omega(X,Z)) -\|\omega_{\nu}(X,Z)\|^{2} = g(\mathcal{Q}_{Z}X, J\omega_{\nu}(X,Z))
\end{equation*}
 for $X \in D$ and $Z \in D^{\perp}.$
\end{lemma}
\begin{proof}
From Gauss equation, we have
\begin{equation*}\label{w10}
\nabla_{Z}JX + \omega(JX,Z) -J(\nabla_{Z}X) - J\omega(X,Z) = \mathcal{P}_{Z}X + \mathcal{Q}_{Z}X
\end{equation*}
Using (\ref{w8}), we infer
\begin{equation*}\label{w11}
\omega(JX,Z)= \mathcal{P}_{Z}X + \mathcal{Q}_{Z}X + J[X(log f)Z] + J\omega(X,Z) - JX(log f)Z .
\end{equation*}
Replace  $X$ by $JX$, we get
\begin{equation*}\label{w12}
-\omega(X,Z)= \mathcal{P}_{Z}JX + \mathcal{Q}_{Z}JX + JX(log f)JZ + J\omega(JX,Z) + X(log f)Z.
\end{equation*}
We can write the above equation as
\begin{equation}\label{w13}
-\omega(X,Z)= \mathcal{P}_{Z}JX + \mathcal{Q}_{Z}JX + JX(log f)JZ + J\omega_{JD^{\perp}}(JX,Z) + J\omega_{\nu}(JX,Z) + X(log f)Z.
\end{equation}
On comparing the tangential components, we obtain
\begin{equation*}\label{w14}
\mathcal{P}_{Z}JX  + J\omega_{JD^{\perp}}(JX,Z)  + X(log f)Z = 0,
\end{equation*}
or
\begin{equation*}\label{w15}
J\omega_{JD^{\perp}}(JX,Z) = -\mathcal{P}_{Z}JX  - X(log f)Z
\end{equation*}
which shows that
\begin{equation}\label{w16}
\omega_{JD^{\perp}}(JX,Z) = J\mathcal{P}_{Z}JX  +  X(log f)JZ,
\end{equation}
for $X \in D$ and $Z \in D^{\perp}$.

Again on comparing the normal components in (\ref{w13}), we have
\begin{equation*}\label{w17}
-\omega(X,Z)= \mathcal{Q}_{Z}JX + JX(log f)JZ  + J\omega_{\nu}(JX,Z)
\end{equation*}
from which we conclude that
\begin{equation*}\label{w18}
\omega(JX,Z)= \mathcal{Q}_{Z}X + X(log f)JZ  + J\omega_{\nu}(X,Z)
\end{equation*}
or
\begin{equation}\label{w19}
\omega(JX,Z) -  J\omega_{\nu}(X,Z)= \mathcal{Q}_{Z}X + X(log f)JZ
\end{equation}
By taking the inner product (\ref{w19}) with $JW$, we get
\begin{equation}\label{w20}
g(\omega_{JD^{\perp}}(JX,Z), JW) = g(\mathcal{Q}_{Z}X,JW) + X(log f)g(JZ,JW)
\end{equation}
Further using (\ref{w16}) in (\ref{w20}), we have
\begin{equation*}\label{w21}
g(\mathcal{P}_{Z}JX, W)  +  X(log f)g(Z, W) = g(\mathcal{Q}_{Z}X,JW) + X(log f)g(Z, W)
\end{equation*}
from which we conclude that
\begin{equation}\label{w22}
g(\mathcal{P}_{Z}JX, W)   = g(\mathcal{Q}_{Z}X,JW)
\end{equation}
Also, by taking the inner product of (\ref{w19}) with $J\omega(X,Z)$, we find
\begin{equation*}\label{w23}
g(\omega(JX,Z),J\omega(X,Z)) - \|\omega_{\nu}(X,Z)\|^{2}   = g(\mathcal{Q}_{Z}X,J\omega_{\nu}(X,Z)).
\end{equation*}
\end{proof}
\begin{theorem}
Let $\mathcal{W} = \mathcal{W}_{T}\times_{f}\mathcal{W}_{\perp}$ be a warped product CR-submanifolds of Bochner Kaehler manifold $\overline{\mathcal{W}}$ with $\mathcal{P}_{D^{\perp}}D \in D$, then the squared norm of second fundamental form of $\mathcal{W}$ in $\overline{\mathcal{W}}$ satisfies the following inequality
\begin{equation*}\label{w24}
\|\omega\|^{2} \geq  \|\mathcal{P}_{D^{\perp}}D\|^{2} + q\|grad_{D}(log f)\|^{2}.
\end{equation*}
\end{theorem}

\begin{proof}
Let $\{X_{1}, ...,  X_{p}, X_{p+1} = JX_{1}, ..., X_{2p}= JX_{p}\}$ be a local orthonormal frame of vector fields on $N_{T}$ and $\{Z_{1}, ..., Z_{q}\}$ be a local orthonormal frame of vector fields on $N_{\perp}$, where $2p+q = n$. Then we have
\begin{eqnarray*}\label{w25}
\nonumber \|\omega\|^{2} = \sum_{i,j=1}^{2p}g(\omega(X_{i},X_{j}), \omega (X_{i},X_{j})) + \sum_{\alpha=1}^{q}\sum_{j=1}^{2p}g(\omega(X_{i},Z_{\alpha}),\omega(X_{i},Z_{\alpha}))\\  +\sum_{\alpha,\beta =1}^{q}g(\omega(Z_{\alpha},Z_{\beta}), \omega(Z_{\alpha},Z_{\beta}))
\end{eqnarray*}
from above equation we can say that
\begin{eqnarray*}\label{w26}
 \|\omega\|^{2} \geq  \sum_{j=1}^{2p}\sum_{\alpha=1}^{q}g(\omega(X_{i},Z_{\alpha}),\omega(X_{i},Z_{\alpha}))
\end{eqnarray*}
Now from (\ref{w16}), we have
\begin{eqnarray*}\label{w27}
 \|\omega\|^{2} \geq \sum_{j=1}^{2p}\sum_{\alpha=1}^{q}g(J\mathcal{P}_{Z_{\alpha}}X_{i} - JX_{i}(log f)JZ_{\alpha}, J\mathcal{P}_{Z_{\alpha}}JX_{i}- JX_{i}(log f)JZ_{\alpha})
\end{eqnarray*}
In view of the assumption $\mathcal{P}_{D^{\perp}}D \in D$, we have
\begin{eqnarray*}\label{w28}
\nonumber\|\omega\|^{2} &\geq& \sum_{j=1}^{2p}\sum_{\alpha=1}^{q}\bigg[g(J\mathcal{P}_{Z_{\alpha}}X_{i} , J\mathcal{P}_{Z_{\alpha}}X_{i}) + g(JX_{i}(log f)JZ_{\alpha}, JX_{i}(log f)JZ_{\alpha})\bigg]
\nonumber \\ &=& \sum_{j=1}^{2p}\sum_{\alpha=1}^{q}\bigg[g(\mathcal{P}_{Z_{\alpha}}X_{i}, \mathcal{P}_{Z_{\alpha}}X_{i}) + (JX_{i}(log f))^{2}g(Z_{\alpha},Z_{\alpha})\bigg]
\nonumber \\ &=&  \|\mathcal{P}_{D^{\perp}}D\|^{2} + \sum_{j=1}^{2p}\|JX_{i}(log f)\|^{2}q
 \\ &=&  \|\mathcal{P}_{D^{\perp}}D\|^{2} + q\|grad_{D}(log f)\|^{2}
\end{eqnarray*}
where $grad_{D}$ denotes the gradient of some function on the distribution $D$.

Thus we have
\begin{eqnarray*}\label{w29}
\|\omega\|^{2} \geq  \|\mathcal{P}_{D^{\perp}}D\|^{2} + q\|grad_{D}(log f)\|^{2}.
\end{eqnarray*}
\end{proof}

\begin{theorem}
Let $\mathcal{W} = \mathcal{W}_{T}\times_{f}\mathcal{W}_{\perp}$ be a compact orientable  warped product CR-submanifold of Bochner Kaehler manifold $\overline{\mathcal{W}}$. If $\mathcal{P}_{D^{\perp}}D \in D$ and $B_{\nabla_{JX_{i}}^{\perp}JZ}X_{i} = B_{\nabla_{X_{i}}^{\perp}JZ}JX_{i}$, then we have $\rho \leq 0$, and the equality holds iff $grad_{D}(log f) = 0.$
\end{theorem}
\begin{proof}
Let $X \in D$, $Z \in D^{\perp}$, then from (\ref{a5}), we have
\begin{eqnarray}\label{w33}
\overline{R} (X, JX,Z,JZ) &=& -2M(X,JX)g(Z,Z) - 2M(Z,JZ)g(X,X)
\end{eqnarray}
 Now Codazzi equation is
\begin{eqnarray*}\label{w34}
\nonumber \bigg[\overline{R}(X,Y)Z\bigg]^{\perp} &=& \bigg \{ \nabla_{X}^{\perp}\omega(Y,Z) - \omega(\nabla_{X}Y,Z) - \omega(Y,\nabla_{X}Z)\bigg \}
\\ \nonumber & & - \bigg \{ \nabla_{Y}^{\perp}\omega(X,Z) - \omega(\nabla_{Y}X,Z) - \omega(X,\nabla_{Y}Z)\bigg \}
\end{eqnarray*}
In view of the last equation we may write
\begin{eqnarray}\label{w35}
\nonumber \overline{R}(X, JX, Z, JZ)  &=& g( \nabla_{X}^{\perp}\omega(JX,Z) - \omega(\nabla_{X}JX,Z) - \omega(JX,\nabla_{X}Z),JZ)
\\ & & - g(\nabla_{JX}^{\perp}\omega(X,Z) - \omega(\nabla_{JX}X,Z) - \omega(X,\nabla_{JX}Z),JZ)
\end{eqnarray}
We now compute each term of (\ref{w35}). First  we have
\begin{eqnarray}\label{w36}
Xg(\omega(JX,Z),JZ) = g(\overline{\nabla}_{X}\omega(JX,Z), JZ) + g(\omega(JX,Z),\overline{\nabla}_{X}JZ)
\end{eqnarray}
 Using Weingarten formula we have
\begin{eqnarray}\label{w37}
g(\overline{\nabla}_{X}^{\perp}\omega(JX,Z), JZ) =  Xg(\omega(JX,Z),JZ) -g(\omega(JX,Z),\overline{\nabla}_{X}JZ)
\end{eqnarray}
Now from (\ref{w19})
\begin{eqnarray}\label{w38}
\omega(JX, Z)  - J\omega_{\nu}(X,Z) = \mathcal{Q}_{Z}X + X(log f)JZ
\end{eqnarray}
Taking the inner product of (\ref{w38}) with $JZ$, we have
\begin{eqnarray}\label{w39}
g(\omega(JX, Z),JZ)  - g(J\omega_{\nu}(X,Z), JZ) = g(\mathcal{Q}_{Z}X, JZ)  + X(log f)g(JZ,JZ)
\end{eqnarray}
Combining (\ref{w22}) and (\ref{w39}), we get
\begin{eqnarray}\label{w41}
g(\omega(JX, Z),JZ) = g(\mathcal{\mathcal{P}}_{Z}JX, Z)  + X(log f)\|Z\|^{2}
\end{eqnarray}
Moreover
\begin{eqnarray*}
g(\omega(JX,Z), JZ) = X(log f)g(Z,Z)
\end{eqnarray*}
Hence we have
\begin{eqnarray}\label{w42}
\nonumber Xg(\omega(JX,Z), JZ) &=& X\bigg\{ X( log f) g(Z,Z) \bigg\}
\\ \nonumber &=& X\big( X(log f)\big) g(Z,Z)+ 2X(log f)g(Z,\nabla_{X}Z)
\\ \nonumber &=& X\big( X(log f)\big)\|Z\|^{2}+ 2\big( X(log f)\big)^{2}\|Z\|^{2}
\\ &=& \bigg \{  X (X(log f))+ 2( X(log f))^{2} \bigg \} \|Z\|^{2}
\end{eqnarray}
 From (\ref{w37}) and (\ref{w42}), we get
\begin{eqnarray}\label{w43}
g(\nabla_{X}^{\perp}\omega(JX,Z), JZ) = \bigg \{  X (X(log f))+ 2( X(log f))^{2} \bigg \} \|Z\|^{2} - g(\omega(JX,Z), \overline{\nabla}_{X}JZ)
\end{eqnarray}
Replacing $X$ by $JX$ in the above equation , we find
\begin{eqnarray}\label{w44}
-g(\nabla_{JX}^{\perp}\omega(X,Z), JZ) = \bigg \{  JX (JX(log f))+ 2( JX(log f))^{2} \bigg \} \|Z\|^{2} + g(\omega(X,Z), \overline{\nabla}_{JX}JZ)
\end{eqnarray}
Also using (\ref{w16}) and $\mathcal{P}_{D^{\perp}} D \in D$, we conclude that
\begin{eqnarray}\label{w45}
g(\omega_{JD^{\perp}}(JX, \nabla_{X}Z), JZ ) =  g(X(log f)J\nabla_{X}Z,JZ) = (X(log f))^{2}g(Z,Z) = (X(log f))^{2}\|Z\|^{2}
\end{eqnarray}
 Replacing $X$ by $JX$ in the above equation, we find
\begin{eqnarray}\label{w46}
g(\omega_{JD^{\perp}}(X, \nabla_{JX}Z), JZ )  = -(JX(log f))^{2}\|Z\|^{2}
\end{eqnarray}
Again using (\ref{w16}), we get
\begin{eqnarray*}\label{w47}
\omega_{JD^{\perp}}( \nabla_{JX}X), Z )  = J\mathcal{P}_{Z}\nabla_{JX}X- J\nabla_{JX}X(log f)JZ
\end{eqnarray*}
or
\begin{eqnarray*}\label{w48}
g(\omega_{JD^{\perp}}( \nabla_{JX}X), JZ )  = g(\mathcal{P}_{Z}\nabla_{JX}X, Z)- J\nabla_{JX}X(log f)\|Z\|^{2}
\end{eqnarray*}
The  above equation can be written as
\begin{eqnarray*}\label{w49}
g(\omega( \nabla_{JX}X), JZ )  = g(\mathcal{P}_{Z}\nabla_{JX}X, Z)- J\nabla_{JX}X(log f))\|Z\|^{2}
\end{eqnarray*}
 But $N_{T}$ is totally geodesic in $\overline{N}$ which implies that $\nabla_{JX}X \in D$. Hence $\mathcal{P}_{Z}J\nabla_{JX}X \in D$. This makes the first term in the above equation zero and hence we have
\begin{eqnarray}\label{w49}
g(\omega( \nabla_{JX}X), JZ )  = - J\nabla_{JX}X(log f))\|Z\|^{2}
\end{eqnarray}
Similarly on replacing $X$ by $JX$ in the above equation, we have
\begin{eqnarray*}\label{w50}
g(\omega( \nabla_{X}JX), JZ )  = - J\nabla_{X}JX(log f))\|Z\|^{2}
\end{eqnarray*}
 Using Gauss equation, the last equation simplifies to
\begin{eqnarray}\label{w51}
g(\omega( \nabla_{X}JX), JZ )  =  \nabla_{X}X(log f)g(Z,Z) + \nabla_{JX}JX(log f)g(Z,Z) - J\nabla_{JX}X(log f)g(Z,Z)
\end{eqnarray}
 Putting (\ref{w43})$\sim$(\ref{w51}) into (\ref{w35}), we get
\begin{eqnarray}\label{w52}
\nonumber \overline{R}(X,JX,Z,JZ )  &=&  \bigg \{ X(X(log f)) + 2(X(log f))^{2} \bigg \}\|Z\|^{2} - g(\omega(JX,Z), \nabla_{X}^{\perp}JZ)
\nonumber \\ & & - \nabla_{X}X(log f)\|Z\|^{2} - \nabla_{JX}JX(log f)\|Z\|^{2} + J\nabla^{JX}X(log f)\|Z\|^{2}
\nonumber \\ & & - (X(log f))^{2}\|Z\|^{2} + \bigg \{ JX(JX(log f)) + 2(JX(log f))^{2} \bigg \}\|Z\|^{2}
\nonumber \\ & & + g(\omega(X,Z), \nabla_{JX}^{\perp}JZ) - J\nabla_{JX}X(log f)\|Z\|^{2}
\end{eqnarray}
From(\ref{w33}) and (\ref{w52})
\begin{eqnarray*}\label{w53}
\noindent& &\nonumber - 2M(X,JX)g(Z,Z) - 2M(Z,JZ)g(X,X) \hspace{4cm}
\nonumber \\ \text{ } = & &\bigg \{ X(X(log f)) + 2(X(log f))^{2} \bigg \}\|Z\|^{2}
\nonumber \\ & & - g(\omega(JX,Z), \nabla_{X}^{\perp}JZ) - \nabla_{X}X(log f)\|Z\|^{2} - \nabla_{JX}JX(log f)\|Z\|^{2}
\nonumber \\ & & - (X(log f))^{2}\|Z\|^{2} + \bigg \{ JX(JX(log f)) + 2(JX(log f))^{2} \bigg \}\|Z\|^{2}
\nonumber \\ & & + g(\omega(X,Z), \nabla_{JX}^{\perp}JZ) - (JX(log f))^{2}\|Z\|^{2}
\end{eqnarray*}
Putting $X = X_{i}$ and taking summation from 1 to $p$, we drive
\begin{eqnarray*}\label{w54}
\noindent & &\nonumber -2\|Z\|^{2}\sum_{i=1}{p}M (X_{i}, JX_{i}) - 2M(Z, JZ)p
\nonumber \\ \text{ } = & & \sum_{i=1}^{p}\bigg \{ X_{i}(X_{i}(log f)) + JX_{i}(JX_{i}(log f)) - \nabla_{X_{i}}X_{i}(log f)) - \nabla_{JX_{i}}JX_{i}(log f)\bigg \}\|Z\|^{2}
\nonumber \\ & & + \sum_{i=1}^{p}\bigg \{ (X_{i}(log f))^{2} + (JX_{i}(log f))^{2}\bigg \}\|Z\|^{2}
\nonumber \\ & & + \sum_{i=1}{p}\bigg[ g(\omega(X_{i},Z), \nabla_{JX_{i}}^{\perp}JZ) - g(\omega(JX_{i},Z), \nabla_{X_{i}}^{\perp}JZ)\bigg]\|Z\|^{2}
\end{eqnarray*}
 from which we have
\begin{eqnarray*}\label{w55}
\noindent \nonumber & &-2\|Z\|^{2}\sum_{i=1}{p} M (X_{i}, JX_{i}) - 2M(Z, JZ)p
\nonumber \\ &  =&  \Delta_{D}(log f)\|Z\|^{2} + \|grad_{D}(log f)\|^{2}\|Z\|^{2}
\nonumber \\  & &+ \sum_{i=1}{p}\bigg[ g(\omega(X_{i},Z), \nabla_{JX_{i}}^{\perp}JZ) - g(\omega(JX_{i},Z), \nabla_{X_{i}}^{\perp}JZ)\bigg]\|Z\|^{2}
\end{eqnarray*}
Using (\ref{a6}) and (\ref{a7}) in the last equation we arrive at
\begin{eqnarray*}\label{w56}
\noindent \nonumber \frac{-1}{n+2}\sum_{i=1}^{p}\bigg[\|Z\|^{2}Ric(X_{i}, X_{i}) + \|X_{i}\|^{2} Ric(Z,Z)\bigg] + \frac{\rho \|X_{i}\|^{2}\|Z\|^{2}}{2(n+1)(n+2)}
\nonumber \\  = \Delta_{D}(log f)\|Z\|^{2} + \|grad_{D}(log f)\|^{2}\|Z\|^{2}
\nonumber \\+  \sum_{i=1}^{p}\bigg[ g(\omega(X_{i},Z), \nabla_{JX_{i}}^{\perp}JZ) - g(\omega(JX_{i},Z), \nabla_{X_{i}}^{\perp}JZ)\bigg]\|Z\|^{2}
\end{eqnarray*}
from which we have
\begin{eqnarray*}\label{w58}
\noindent \nonumber \frac{-1}{n+2}\bigg[\|Z\|^{2}\sum_{i=1}^{p}Ric(X_{i}, X_{i}) + p Ric(Z,Z)\bigg] + \frac{\rho p\|Z\|^{2}}{2(n+1)(n+2)}
\nonumber \\  = \Delta_{D}(log f)\|Z\|^{2} + \|grad_{D}(log f)\|^{2}\|Z\|^{2}
\nonumber \\+  \sum_{i=1}^{p}\bigg[ g(B_{\nabla_{JX_{i}}^{\perp} JZ}X_{i}, Z) - g(B_{\nabla_{X_{i}}^{\perp} JZ}JX_{i}, Z)\bigg]\|Z\|^{2}
\end{eqnarray*}
Since by assumption, we have $B_{\nabla_{JX_{i}}^{\perp} JZ}X_{i} = B_{\nabla_{X_{i}}^{\perp} JZ}JX_{i} $, then (\ref{w58}) becomes
\begin{eqnarray*}\label{w59}
\noindent \nonumber \frac{-1}{n+2}\bigg[\sum_{i=1}^{p}Ric(X_{i}, X_{i}) +\frac{p}{\|Z\|^{2}}  Ric(Z,Z)\bigg] + \frac{p \rho \|Z\|^{2}}{2(n+1)(n+2)}
\nonumber \\  = \Delta_{D}(log f) + \|grad_{D}(log f)\|^{2}
\end{eqnarray*}
Integrating both sides and using Green's equation, the last equation simplifies to
\begin{eqnarray}\label{w60}
\nonumber\frac{-1}{n+2}\int\bigg[\sum_{i=1}^{p}Ric(X_{i}, X_{i}) +\frac{p}{\|Z\|^{2}}  Ric(Z,Z)\bigg]dv + \int\frac{p \rho \|Z\|^{2}}{2(n+1)(n+2)}dv
\\ = \int \|grad_{D}(log f)\|^{2}dv
\end{eqnarray}
Similarly we have
\begin{eqnarray}\label{w61}
\frac{-1}{n+2}\int\bigg[\sum_{i=1}^{p}Ric(JX_{i}, JX_{i}) +\frac{p}{\|Z\|^{2}}  Ric(Z,Z)\bigg]dv + \int\frac{p \rho \|Z\|^{2}}{2(n+1)(n+2)}dv \nonumber \\ = \int \|grad_{D}(log f)\|^{2}dv
\end{eqnarray}
Adding  (\ref{w60}) and (\ref{w61}), we find
\begin{eqnarray*}\label{w62}
\frac{-1}{n+2}\int\bigg[\sum_{i=1}^{p}Ric(X_{i}, X_{i})+ \sum_{i=1}^{p}Ric(JX_{i}, JX_{i}) +\frac{2p}{\|Z\|^{2}}  Ric(Z,Z)\bigg]dv  \nonumber \\ + \int\frac{2p \rho \|Z\|^{2}}{2(n+1)(n+2)} dv = 2\int \|grad_{D}(log f)\|^{2}dv
\end{eqnarray*}
 from which we have
\begin{eqnarray*}\label{w63}
\frac{-1}{n+2}\int\bigg[\rho_{D} +\frac{2p}{\|Z\|^{2}}  Ric(Z,Z)\bigg]dv + \int\frac{p \rho \|Z\|^{2}}{(n+1)(n+2)} dv = 2\int \|grad_{D}(log f)\|^{2} dv
\end{eqnarray*}
 where $\rho _{D}$ is the scalar curvature of distribution $D$.
  Further replacing $Z$ by $Z_{\alpha}$ and taking summation from $1$ to $q$ on both sides.
  As
\begin{eqnarray*}\label{w64}
q\int \|grad_{D}(log f)\|^{2}dv \geq 0
\end{eqnarray*}
we conclude  that
\begin{eqnarray*}\label{w65}
\frac{-1}{n+2}\int\bigg[q\rho_{D} +2p\rho_{D^{\perp}}\bigg]dv + \int\frac{pq^{2} \rho}{(n+1)(n+2)} dv\geq 0
\end{eqnarray*}
This shows  that
\begin{eqnarray*}\label{w66}
\frac{pq^{2}}{(n+1)(n+2)}\int \rho dv\geq \frac{1}{n+2}\int\bigg[q\rho_{D} +2p\rho_{D^{\perp}}\bigg]dv
\end{eqnarray*}
or
\begin{eqnarray}\label{w67}
\int \rho dv \geq (n+1)\int\bigg[\frac{\rho_{D}}{pq} +\frac{2(n+1)}{q^{2}}\rho_{D^{\perp}}\bigg]dv
\end{eqnarray}
Thus we have
\begin{eqnarray*}\label{w68}
\int\bigg[\rho_{D} + \rho_{D^{\perp}}\bigg]dv \geq \int \bigg[\frac{(n+1)}{pq}\rho_{D} +\frac{2(n+1)}{q^{2}} \rho_{D^{\perp}}\bigg]dv
\end{eqnarray*}
From we have the following observations.
Either $(n+1) \leq pq$ and  $2(n+1) \leq q^{2}$ or $\rho_{D} \leq 0$ and $\rho_{D^{\perp}} \leq 0$ that id $\rho = \rho_{D} + \rho_{D^{\perp}} \leq 0$. Equality holds if and only if either $(n+1) = pq$ and $2(n+1) = q^{2}$ or $grad_{D}(log f) = 0.$
\end{proof}
\hspace{6cm}{\bf{References}}


\begin{thebibliography}{11}
\markboth{ M. A. Lone, M. Jamali, M. H. Shahid}{Some inequalities of submanifolds of Bochner Kaehler manifold}
\bibitem{biha11} Bejancu A.:
\newblock {CR-submanifolds of a Kaehler manifold. I},
\newblock Proc. Amer. Math. Soc.,  69, 1978, 135-142.
\bibitem{bishopwrap} Bishop R. L., O’Neill B.:
\newblock \emph { Manifolds of negative curvature},
\newblock  Trans. Amer. Math. Soc. 145, 1969, 01-49.
\bibitem{biha2} Chen B. Y.:
\newblock \emph{A general inequality for submanifolds in complex space forms and its applications},
\newblock Arch. Math., 67, 1996, 519-528.
\bibitem{biha3} Chen B. Y.:
\newblock \emph {Geometry of Slant submanifolds},
\newblock Katholieke Universitiet Leuven, 1990.
\bibitem{chen warp}  Chen B. Y.:
\newblock \emph {Geometry of warped product CR-submanifolds in Kaehler manifolds},
\newblock Monatsh. Math. 133, 2001, 177-195.
\bibitem{biha4} Chen B. Y.:
\newblock \emph {Some pinching and classification theorems for minimal submanifolds},
\newblock Arch. math., 60, 1993, 568-578.
\bibitem{biha5} Cioroboiu D,  Oiaga A.:
\newblock \emph {B. Y. Chen inequalities for  slant submanifolds in Sasakian  space forms},
\newblock Rendiconti del Circolo Matematico di Palermo 10, 2004, 367-381.
\bibitem{biha6} Kim J. S.,  Song Y. M.,  Tripathi M. M.:
\newblock \emph {B. Y. Chen inequalities for  submanifolds in generilized complex space forms},
\newblock Bull. Korean Math. Soc., 40(3), 2003, 411-423.
\bibitem{biha1} Oiaga A.,  Mihai I.:
\newblock \emph {B. Y. Chen inequalities for slant submanifolds in complex space forms},
\newblock Demonstratio Math., 32(4), 1990, 835-846.
\bibitem{biha7} Shahid M. H.,  Husain S. I.:
\newblock \emph {CR-submanifolds of a Bochner Kaehler manifold},
\newblock Indian J. Pure. and Applied Math., 18, 1987, 605-610.


 \bibitem{olteanu} {A. Olteanu},
 \newblock \emph { A general inequality for doubly warped product submanifolds}, Math. J. Okayama
Univ., 52(2010), 133-142

A. OLTEANU, CR-doubly warped product submanifolds in Sasakian space forms, Bulletin of the
Transilvania University of Brasov, 1 (50), III-2008, 269–278


Munteanu, M. I., Warped product contact CR-submanifolds of Sasakian space form, Publ.
Math. Debrecen, 66(6)(2005), 75-120.


] Chen, B. Y., Geometry of warped product CR-submanifolds in Kaehler manifold II, Monatsh.
Math., 134(2001), 103-119

\end{thebibliography}
\end{document}